\documentclass[12pt]{amsart}
\usepackage{amssymb,latexsym,amsmath,epsfig,amsthm,mathrsfs}
\usepackage{amsfonts}
\usepackage{amscd}
\usepackage{dsfont}
\usepackage{bbm}
\usepackage{rotating}
\usepackage{graphicx}
\usepackage{amssymb}
\usepackage{lineno}
\usepackage{enumitem}
\usepackage{cite}
\usepackage[top=3cm, bottom=3cm, left=3.5cm, right=3.5cm]{geometry}
\usepackage[usenames]{color}
\usepackage[colorlinks=true,
linkcolor=blue,
filecolor=blue,
citecolor=blue]{hyperref}

\newtheorem{theorem}{Theorem}[section]

\newtheorem{corollary}{Corollary}[section]

\theoremstyle{definition}
\newtheorem{definition}[theorem]{Definition}

\theoremstyle{remark}
\newtheorem{remark}[theorem]{Remark}

\numberwithin{equation}{section}
\numberwithin{theorem}{section}
\begin{document}

\title[Inverse problems for certain subsequence sums in integers]{Inverse problems for certain subsequence sums in integers}


\author[J Bhanja]{Jagannath Bhanja*}
\address{Department of Mathematics, Indian Institute of Technology Roorkee, Uttarakhand, 247667, India}
\email{jbhanja90@gmail.com}
\thanks{$^{*}$Research supported by the Ministry of Human Resource Development, India.}

\author[R K Pandey]{Ram Krishna Pandey}
\address{Department of Mathematics, Indian Institute of Technology Roorkee, Uttarakhand, 247667, India}
\email{ramkpandey@gmail.com}





\begin{abstract}
Let $A$ be a nonempty finite set of $k$ integers. Given a subset $B$ of $A$, the sum of all
elements of $B$, denoted by $s(B)$, is called the subset sum of $B$. For a nonnegative integer $\alpha$ ($\leq k$), let
\[\Sigma_{\alpha} (A):=\{s(B): B \subset A, |B|\geq \alpha\}.\]
Now, let $\mathcal{A}=(\underbrace{a_{1},\ldots,a_{1}}_{r_{1}~\text{copies}},
\underbrace{a_{2},\ldots,a_{2}}_{r_{2}~\text{copies}},\ldots,
\underbrace{a_{k},\ldots,a_{k}}_{r_{k}~\text{copies}})$ be a finite sequence of integers with $k$ distinct terms, where $r_{i}\geq 1$ for $i=1,2,\ldots,k$. Given a subsequence $\mathcal{B}$ of $\mathcal{A}$, the sum of all terms of $\mathcal{B}$, denoted by $s(\mathcal{B})$, is called the subsequence sum of $\mathcal{B}$. For $0\leq \alpha \leq \sum_{i=1}^{k} r_{i}$, let \[\Sigma_{\alpha} (\bar{r},\mathcal{A}):=\left\{s(\mathcal{B}): \mathcal{B}~\text{is a subsequence of}~\mathcal{A}~\text{of length} \geq \alpha \right\},\]
where $\bar{r}=(r_{1},r_{2},\ldots,r_{k})$. Very recently, Balandraud obtained the minimum cardinality of $\Sigma_{\alpha} (A)$ in finite fields. Motivated by Baladraud's work, we find the minimum cardinality of $\Sigma_{\alpha}(A)$ in the group of integers. We also determine the structure of the finite set $A$ of integers for which $|\Sigma_{\alpha} (A)|$ is minimal. Furthermore, we generalize these results of subset sums to the subsequence sums $\Sigma_{\alpha} (\bar{r},\mathcal{A})$. As special cases of our results we obtain some already known results for the usual subset and subsequence sums.
\end{abstract}

\maketitle

\section{Introduction}
Let $A$ be a nonempty finite set of $k$ integers. Given a subset $B$ of $A$, the sum of all elements of $B$ is called the \emph{subset sum} of $B$, and it is denoted by $s(B)$. In other words, $s(B):=\sum_{b\in B} b$. Let $\Sigma (A)$ be the set of all subset sums of $A$, i.e.,
\[
\Sigma (A):=\{s(B): B \subset A\},
\]
where we assume that $s(\emptyset)=0$.

The subsequence sum of a given sequence of integers is defined in a similar way. Let $\mathcal{A}=(\underbrace{a_{1},\ldots,a_{1}}_{r_{1}~\text{copies}},
\underbrace{a_{2},\ldots,a_{2}}_{r_{2}~\text{copies}},\ldots,
\underbrace{a_{k},\ldots,a_{k}}_{r_{k}~\text{copies}})$ be a finite sequence of integers with $k$ distinct terms $a_{1}, a_{2},\ldots,a_{k}$, where $r_{i}\geq 1$ for $i=1,2,\ldots,k$. For the convenience, we denote this sequence by $(a_{1},a_{2},\ldots,a_{k})_{\bar{r}}$, where $\bar{r}=(r_{1},r_{2},\ldots,r_{k})$ be the ordered $k$-tuple. Given a subsequence $\mathcal{B}$ of $\mathcal{A}$, the sum of all terms of $\mathcal{B}$ is called the \emph{subsequence sum} of $\mathcal{B}$, and it is denoted by $s(\mathcal{B})$. In other words, $s(\mathcal{B}):=\sum_{b\in \mathcal{B}} b$. Let $\Sigma (\bar{r},\mathcal{A})$ be the set of all subsequence sums of $\mathcal{A}$, i.e.,
\[
\Sigma (\bar{r},\mathcal{A}):=\{s(\mathcal{B}): \mathcal{B}~\text{is a subsequence of}~\mathcal{A}\}.
\]
If $r_{i}=r$ for $i=1,2,\ldots,k$, then we write $\Sigma (r,\mathcal{A})$ instead of $\Sigma (\bar{r},\mathcal{A})$.

The subset and subsequence sums are fundamental objects in additive number theory. These sumsets appear, quite often, in the study of the zero-sum constants such as \emph{Noether number}, \emph{Davenport constant} and some variations of these constants \cite{cziszter, ordaz, schmid}. In order to find these zero-sum constants, sometimes it is necessary to bound the subset and subsequence sums. In these problems, apart from the regular subset and subsequence sums, the subset and subsequence sums with some restriction on the number of terms have been appeared several times (see \cite{bollobas, erdos, gao04, gao08, girard, grynkiewicz1, grynkiewicz2, hamidoune}). The first formal study of these subset sums with some restriction is due to Balandraud \cite{eric17} in 2017. He obtained the minimum cardinality of these subset sums in finite fields. In this paper, we study the same subset sums with some restriction, but in the group of integers. We also study the analogous subsequence sums in the group of integers.

\begin{definition} Let $A$ be a nonempty finite set of $k$ integers. Let $0\leq \alpha \leq k$ be an integer. We define $\Sigma_{\alpha} (A)$ to be the set of subset sums of all subsets of $A$ that are of the size at least $\alpha$, and $\Sigma^{\alpha} (A)$ to be the set of subset sums of all subsets of $A$ that are of the size at most $k-\alpha$. More precisely,
\[
\Sigma_{\alpha} (A):=\{s(B): B \subset A, |B|\geq \alpha\},
\]
and
\[
\Sigma^{\alpha} (A):=\{s(B): B \subset A, |B|\leq k-\alpha\}.
\]
\end{definition}

It is easy to see that these subset sums have the following properties:
\begin{itemize}
\item If $\alpha=0$, then $\Sigma_{0} (A)=\Sigma^{0} (A)=\Sigma (A)$.

\item For every $0\leq \alpha \leq k$, one has $\Sigma_{\alpha} (A)=s(A)-\Sigma^{\alpha} (A)$. Thus, $|\Sigma_{\alpha} (A)|=|\Sigma^{\alpha} (A)|$.

\item If $\alpha \leq \alpha'$, then $\Sigma_{\alpha'} (A)\subset \Sigma_{\alpha} (A)$ and $\Sigma^{\alpha'} (A)\subset \Sigma^{\alpha} (A)$.
\end{itemize}

\begin{definition} Let $\mathcal{A}=(a_{1},a_{2},\ldots,a_{k})_{\bar{r}}$ be a nonempty finite sequence of integers with $k$ distinct terms and repetition $\bar{r}=(r_{1},r_{2},\ldots,r_{k})$. Let $0\leq \alpha \leq \sum_{i=1}^{k}r_{i}$ be an integer. We define $\Sigma_{\alpha} (\bar{r},\mathcal{A})$ to be the set of subsequence sums of all subsequences of $\mathcal{A}$ that are of the length at least $\alpha$, and $\Sigma^{\alpha} (\bar{r},\mathcal{A})$ to be the set of subsequence sums of all subsequences of $\mathcal{A}$ that are of the length at most $\sum_{i=1}^{k}r_{i}-\alpha$. More precisely,
\[
\Sigma_{\alpha} (\bar{r},\mathcal{A}):=\left\{s(\mathcal{B}): \mathcal{B}~\text{is a subsequence of}~\mathcal{A}~\text{of length}\geq \alpha \right\},
\]
and
\[
\Sigma^{\alpha} (\bar{r},\mathcal{A}):=\{s(\mathcal{B}): \mathcal{B}~\text{is a subsequence of}~\mathcal{A}~\text{of length}\leq \sum_{i=1}^{k}r_{i}-\alpha \}.
\]
\end{definition}

These subsequence sums also satisfy similar properties as that satisfied by the aforementioned subset sums:
\begin{itemize}
\item If $\alpha=0$, then $\Sigma_{0} (\bar{r},\mathcal{A})=\Sigma^{0} (\bar{r},\mathcal{A})=\Sigma (\bar{r},\mathcal{A})$.

\item For every $0\leq \alpha \leq \sum_{i=1}^{k}r_{i}$, one has $\Sigma_{\alpha} (\bar{r},\mathcal{A})=s(\mathcal{A})-\Sigma^{\alpha} (\bar{r},\mathcal{A})$. Thus, $|\Sigma_{\alpha} (\bar{r},\mathcal{A})|=|\Sigma^{\alpha} (\bar{r},\mathcal{A})|$.

\item If $\alpha \leq \alpha'$, then $\Sigma_{\alpha'} (\bar{r},\mathcal{A})\subset \Sigma_{\alpha} (\bar{r},\mathcal{A})$ and $\Sigma^{\alpha'} (\bar{r},\mathcal{A})\subset \Sigma^{\alpha} (\bar{r},\mathcal{A})$.
\end{itemize}
If $r_{i}=r$ for $i=1,2,\ldots,k$, then we write $\Sigma_{\alpha}(r,\mathcal{A})$ instead of   $\Sigma_{\alpha}(\bar{r},\mathcal{A})$ and $\Sigma^{\alpha}(r,\mathcal{A})$ instead of $\Sigma^{\alpha}(\bar{r},\mathcal{A})$.

The \emph{direct problem} for the subset sums $\Sigma_{\alpha}(A)$ is to find the minimum cardinality of $\Sigma_{\alpha}(A)$ in terms of number of elements in the set $A$ and $\alpha$. The \emph{inverse problem} for $\Sigma_{\alpha}(A)$ is to determine the structure of the finite set $A$ for which $|\Sigma_{\alpha} (A)|$ is minimal. Similarly, the direct problem for the subsequence sums $\Sigma_{\alpha} (\bar{r},\mathcal{A})$ is to find the minimum cardinality of $\Sigma_{\alpha} (\bar{r},\mathcal{A})$ in terms of number of distinct terms in the sequence $\mathcal{A}$ and $\alpha$. The inverse problem for $\Sigma_{\alpha} (\bar{r},\mathcal{A})$ is to determine the structure of the finite sequence $\mathcal{A}$ for which $|\Sigma_{\alpha}(\bar{r},\mathcal{A})|$ is minimal.

The direct and inverse problems for the regular subset sums $\Sigma (A)$ in integers have been first studied by Nathanson \cite{nathu95} in 1995. Later, in 2015, Mistri and Pandey \cite {mistri15} (see also \cite{mistri16}) generalized Nathanson's results to the subsequence sums $\Sigma (\bar{r},\mathcal{A})$ in two separate cases; namely, (i) the sequence $\mathcal{A}$ contains only positive integers (ii) the sequence $\mathcal{A}$ contains only nonnegative integers with $0\in \mathcal{A}$. Very recently, Jiang and Li \cite{jiang} have settled the remaining case, i.e., where the sequence $\mathcal{A}$ contains positive integers, negative integers and/or zero.

In this paper, we first solve the direct and inverse problems for the subset sums $\Sigma_{\alpha} (A)$ in Section \ref{subsetsum}. Then, we solve the direct and inverse problems for the subsequence sums $\Sigma_{\alpha} (\bar{r},\mathcal{A})$ in Section \ref{subsequencesum}. As corollaries of our results for subset sums $\Sigma_{\alpha} (A)$ we obtain the direct and inverse theorems of Nathanson \cite{nathu95} on usual subset sums $\Sigma (A)$. Similarly, as corollaries of our results for subsequence sums $\Sigma_{\alpha} (\bar{r},\mathcal{A})$ we obtain the direct and inverse theorems of Mistri and Pandey \cite{mistri15} on usual  subsequence sums $\Sigma (\bar{r},\mathcal{A})$.

In our study, we consider two separate cases; namely, (i) the set $A$ (or sequence $\mathcal{A}$) contains only positive integers (ii) the set $A$ (or sequence $\mathcal{A}$) contains only nonnegative integers with $0\in A$ (or $0\in \mathcal{A}$).

For any two integers $a$, $b$ $(b\geq a)$, we write $[a,b]$ for the set $\{a,a+1,\ldots,b\}$, and the sequence interval $[a,b]_{r}$ for the sequence $(a,a+1,\ldots,b)_{r}$. For a set $A$, and for an integer $c$, we let $c*A=\{ca:a\in A\}$. Similarly, for a sequence $\mathcal{A}=(a_{1},a_{2},\ldots,a_{k})_{\bar{r}}$, and for a positive integer $c$, we let $c*\mathcal{A}=(ca_{1},ca_{2},\ldots,ca_{k})_{\bar{r}}$. Finally, we assume that $\sum_{j=a}^{b} f(j)=0$ if $b<a$, for any $f$.

\section{Direct and inverse problems for subset sums}\label{subsetsum}
\subsection{Direct problem}

\begin{theorem}\label{thm2.1}
Let $k\geq 1$ and $0 \leq \alpha \leq k$. If $A$ is a set of $k$ positive integers, then
\begin{equation}\label{eqn2.1}
|\Sigma_{\alpha} (A)|\geq \frac{k(k+1)}{2}-\frac{\alpha(\alpha+1)}{2}+1.
\end{equation}
The lower bound in (\ref{eqn2.1}) is best possible.
\end{theorem}

\begin{proof}
Let $A=\{a_{1},a_{2},\ldots,a_{k}\}$, where $0<a_{1}<a_{2}<\cdots<a_{k}$. We prove (\ref{eqn2.1}) by induction on $\alpha$. If $\alpha=k$, then $\Sigma_{\alpha} (A)=\{a_{1}+a_{2}+\cdots+a_{k}\}$, and hence $|\Sigma_{\alpha} (A)|=1$. This satisfies (\ref{eqn2.1}).

Now, assume that (\ref{eqn2.1}) holds for $\alpha=k-j$ for some $j=0,1,\ldots,k-1$. We show that (\ref{eqn2.1}) also holds for $\alpha-1$. Note that, the smallest element of $\Sigma_{\alpha} (A)$ is $a_{1}+a_{2}+\cdots+a_{\alpha}$. So, the $\alpha$ distinct sums obtained by removing exactly one summand from $a_{1}+a_{2}+\cdots+a_{\alpha}$, all appear in $\Sigma_{\alpha-1} (A)$ and not in $\Sigma_{\alpha} (A)$. As $\Sigma_{\alpha} (A) \subset \Sigma_{\alpha-1} (A)$, we deduce that
\begin{equation}\label{eqn2.2}
|\Sigma_{\alpha-1} (A)| \geq |\Sigma_{\alpha} (A)|+\alpha.
\end{equation}
Therefore, by induction
\begin{align*}
|\Sigma_{\alpha-1} (A)| &\geq \frac{k(k+1)}{2}-\frac{\alpha(\alpha+1)}{2}+1+\alpha
= \frac{k(k+1)}{2}-\frac{(\alpha-1)\alpha}{2}+1.
\end{align*}
Hence, (\ref{eqn2.1}) holds for $\alpha=0,1,\ldots,k$.

Next, we show that the lower bound in (\ref{eqn2.1}) is best possible. Let $k\geq 2$ and $A=[1,k]$. Then
\begin{align*}
\Sigma^{\alpha} (A) & \subset \left[0,k+(k-1)+\cdots+(\alpha+1)\right]
= \left[0,\frac{k(k+1)}{2}-\frac{\alpha(\alpha+1)}{2}\right].
\end{align*}
Therefore,
\[|\Sigma_{\alpha} (A)|=|\Sigma^{\alpha} (A)|\leq \frac{k(k+1)}{2}-\frac{\alpha(\alpha+1)}{2}+1.\]
This together with (\ref{eqn2.1}) gives
\[|\Sigma_{\alpha} (A)|=\frac{k(k+1)}{2}-\frac{\alpha(\alpha+1)}{2}+1.\]
This completes the proof of theorem.
\end{proof}

\begin{corollary}\label{cor2.1}
Let $k\geq 2$ and $0 \leq \alpha \leq k$. If $A$ is a set of $k$ nonnegative integers with $0\in A$, then
\begin{equation}\label{eqn2.3}
|\Sigma_{\alpha} (A)|\geq \frac{(k-1)k}{2}-\frac{(\alpha-1)\alpha}{2}+1.
\end{equation}
The lower bound in (\ref{eqn2.3}) is best possible.
\end{corollary}

\begin{proof}
Let $A=\{a_{1},a_{2},\ldots,a_{k}\}$, where $0=a_{1}<a_{2}<\cdots<a_{k}$. Set $A'=A\setminus\{a_{1}\}$. So, $A'$ is a set of $k-1$ positive integers. It is easy to see that if $\alpha=0$, then
\begin{equation}\label{equation1}
\Sigma^{0}(A)=\Sigma^{0}(A'),
\end{equation}
and if $\alpha\geq 1$, then
\begin{equation}\label{equation2}
\Sigma^{\alpha}(A)=\Sigma^{\alpha-1}(A').
\end{equation}
Hence, by Theorem \ref{thm2.1}, we have
\begin{align*}
|\Sigma_{0} (A)| &= |\Sigma^{0} (A)|=|\Sigma^{0} (A')| \geq \frac{(k-1)k}{2}+1,
\end{align*}
and for $\alpha\geq 1$, we have
\begin{equation*}
|\Sigma_{\alpha}(A)|=|\Sigma^{\alpha}(A)|=|\Sigma^{\alpha-1}(A')|\geq \frac{(k-1)k}{2}-\frac{(\alpha-1)\alpha}{2}+1.
\end{equation*}
This satisfies (\ref{eqn2.3}).

Next, we show that the lower bound in (\ref{eqn2.3}) is best possible. Let $k\geq 3$, and $A=[0,k-1]$. Then
\begin{align*}
\Sigma^{\alpha} (A) & \subset \left[0,(k-1)+(k-2)+\cdots+\alpha\right]
= \left[0,\frac{(k-1)k}{2}-\frac{(\alpha-1)\alpha}{2}\right].
\end{align*}
Therefore,
\[|\Sigma_{\alpha} (A)|=|\Sigma^{\alpha} (A)|\leq \frac{(k-1)k}{2}-\frac{(\alpha-1)\alpha}{2}+1.\]
This together with (\ref{eqn2.3}) gives
\[|\Sigma_{\alpha} (A)|=\frac{(k-1)k}{2}-\frac{(\alpha-1)\alpha}{2}+1.\]
This completes the proof of the corollary.
\end{proof}

As a consequence of Theorem \ref{thm2.1} and Corollary \ref{cor2.1}, for $\alpha=0$, we get the following corollary.

\begin{corollary}[See \cite{nathu95}; Theorem 3]\label{cor2.2}
Let $k\geq 2$. If $A$ is a set of $k$ positive integers, then
\begin{equation}\label{eqn2.4}
|\Sigma (A)|\geq \frac{k(k+1)}{2}+1.
\end{equation}
If $A$ is a set of $k$ nonnegative integers with $0\in A$, then
\begin{equation}\label{eqn2.5}
|\Sigma (A)|\geq \frac{(k-1)k}{2}+1.
\end{equation}
The lower bounds in (\ref{eqn2.4}) and (\ref{eqn2.5}) are best possible.
\end{corollary}

\subsection{Inverse problem}
\begin{remark}\label{remark1}
Not all extremal sets (i.e., those sets for which equality holds in (\ref{eqn2.1})) are of the form $d*[1,k]$. Here are some examples:
\begin{enumerate}[label=(\roman*)]
  \item Let $A=\{a_{1},a_{2},\ldots,a_{k}\}$ be a set of $k$ positive integers with $0<a_{1}<a_{2}<\cdots<a_{k}$. If $\alpha=k$, then $\Sigma_{k}(A)=\{a_{1}+a_{2}+\cdots+a_{k}\}$, and hence $|\Sigma_{k}(A)|=1$. Similarly, if $\alpha=k-1$, then $\Sigma_{k-1}(A)=\{a_{1}+\cdots+a_{k-1},a_{1}+\cdots+a_{k-2}+a_{k},\ldots,
      a_{2}+\cdots+a_{k},a_{1}+\cdots+a_{k}\}$, and hence $|\Sigma_{k-1}(A)|=k+1$. Thus, every set of $k$ positive integers is an extremal set for $\alpha=k-1~\text{and}~k$.

  \item Let $A=\{a_{1},a_{2}\}$, where $0<a_{1}<a_{2}$. The cases $\alpha=1~\text{and}~2$ are covered in (1), so we let $\alpha=0$. Then $\Sigma_{0}(A)=\{0,a_{1},a_{2},a_{1}+a_{2}\}$, and hence $|\Sigma_{0}(A)|=4$. So, equality holds in (\ref{eqn2.1}). Hence, every set of two positive integers is an extremal set for every $\alpha$.

  \item Let $A=\{a_{1},a_{2},a_{3}\}$, where $0<a_{1}<a_{2}<a_{3}$. Since, the cases $\alpha=2~\text{and}~3$ are covered in (1) we let $\alpha\leq 1$. First, let $\alpha=0$. Then $\Sigma_{0}(A)=\{0,a_{1},a_{2},a_{3},a_{1}+a_{2},a_{1}+a_{3},a_{2}+a_{3},a_{1}+a_{2}+a_{3}\}$,
      where $0<a_{1}<a_{2}<a_{1}+a_{2}<a_{1}+a_{3}<a_{2}+a_{3}<a_{1}+a_{2}+a_{3}$. If equality holds in (\ref{eqn2.1}), i.e., $|\Sigma_{0} (A)|=7$, then $a_{3}=a_{1}+a_{2}$. Hence, $A=\{a_{1},a_{2},a_{1}+a_{2}\}$.

      Next, let $\alpha=1$ and $|\Sigma_{1}(A)|=6$. Since $\Sigma_{1}(A)=\Sigma_{0}(A)\setminus\{0\}$, by the same argument we get $A=\{a_{1},a_{2},a_{1}+a_{2}\}$.

      Hence, $A=\{a_{1},a_{2},a_{1}+a_{2}\}$ with $0<a_{1}<a_{2}$ is an extremal set for every $\alpha$.
\end{enumerate}
\end{remark}

\begin{theorem}\label{thm2.2}
Let $k\geq 4$ and $0\leq \alpha \leq k-2$. Let $A$ be a set of $k$ positive integers such that
\begin{equation}\label{equation5}
|\Sigma_{\alpha} (A)|=\frac{k(k+1)}{2}-\frac{\alpha(\alpha+1)}{2}+1.
\end{equation}
Then $A=d*[1,k]$ for some positive integer $d$.
\end{theorem}

\begin{proof}
Let $A=\{a_{1},a_{2},\ldots,a_{k}\}$, where $0<a_{1}<a_{2}<\cdots<a_{k}$. We prove the
theorem by induction on $\alpha$. First, let $\alpha=k-2$. It is easy to see that
\begin{equation}\label{eqn2.6}
\{0,a_{1},a_{2}\} \bigcup_{i=1}^{k-2} \{a_{i}+a_{i+1},a_{i}+a_{i+2}\} \bigcup \{a_{k-1}+a_{k}\} \subset \Sigma^{k-2} (A).
\end{equation}
Since $|\Sigma^{k-2} (A)|=|\Sigma_{k-2} (A)|=2k$, it is clear that $\Sigma^{k-2} (A)$ contains precisely the integers listed in (\ref{eqn2.6}). Now, for $i=1,2,\ldots,k-3$, consider the integers of the form $a_{i}+a_{i+3}$. Clearly,
\[a_{i}+a_{i+2}<a_{i}+a_{i+3}<a_{i+1}+a_{i+3}.\]
Thus, (\ref{eqn2.6}) implies that $a_{i}+a_{i+3}=a_{i+1}+a_{i+2}$. In other words,
\begin{equation}\label{equation6}
a_{i+3}-a_{i+2}=a_{i+1}-a_{i}~~\text{for}~i=1,2,\ldots,k-3.
\end{equation}
Similarly, since $a_{2}<a_{3}<a_{4}<a_{1}+a_{4}=a_{2}+a_{3}$, (\ref{eqn2.6}) implies that $a_{3}=a_{1}+a_{2}$ and $a_{4}=a_{1}+a_{3}$. That is $a_{4}-a_{3}=a_{3}-a_{2}=a_{1}$. This together with (\ref{equation6}) gives  $a_{k}-a_{k-1}=a_{k-1}-a_{k-2}=\cdots=a_{2}-a_{1}=a_{1}$. Hence, the theorem holds for $\alpha=k-2$.

Suppose that, the theorem holds for $\alpha=k-j$ for some $j=2,3,\ldots,k-1$. We show that the theorem also holds for $\alpha-1$. Let $|\Sigma_{\alpha-1} (A)|=\frac{k(k+1)}{2}-\frac{(\alpha-1)\alpha}{2}+1$. From (\ref{eqn2.2}), we get
\[|\Sigma_{\alpha} (A)| \leq |\Sigma_{\alpha-1} (A)|-\alpha=\frac{k(k+1)}{2}-\frac{(\alpha-1)\alpha}{2}+1-\alpha=\frac{k(k+1)}{2}-\frac{\alpha(\alpha+1)}{2}+1.\]
This together with (\ref{eqn2.1}) gives $|\Sigma_{\alpha} (A)|=\frac{k(k+1)}{2}-\frac{\alpha(\alpha+1)}{2}+1$. Thus, by induction hypothesis $A=a_{1}*[1,k]$. Hence, the theorem holds for $\alpha=0,1,2\ldots,k$. This completes the proof of the theorem.
\end{proof}

\begin{remark}\label{remark2}
As a corollary of Remark \ref{remark1}, we get the following sets which are extremal sets but not of the form $d*[0,k-1]$.
\begin{enumerate}[label=(\roman*)]
  \item Every set $A=\{0,a_{1},a_{2},\ldots,a_{k-1}\}$ with $0<a_{1}<a_{2}<\cdots<a_{k-1}$ is an extremal set for $\alpha=k-1~\text{and}~k$.

  \item Every set $A=\{0,a_{1},a_{2}\}$ with $0<a_{1}<a_{2}$ is an extremal set for every $\alpha$.

  \item Every set $A=\{0,a_{1},a_{2},a_{1}+a_{2}\}$ with $0<a_{1}<a_{2}$ is an extremal set for every $\alpha$.
\end{enumerate}
\end{remark}

\begin{corollary}\label{cor2.3}
Let $k\geq 5$ and $0\leq \alpha \leq k-2$. Let $A$ be a set of $k$ nonnegative integers with $0\in A$ such that
\begin{equation}\label{equation3}
|\Sigma_{\alpha} (A)|=\frac{(k-1)k}{2}-\frac{(\alpha-1)\alpha}{2}+1.
\end{equation}
Then $A=d*[0,k-1]$ for some positive integer $d$.
\end{corollary}

\begin{proof} Let $A=\{a_{1},a_{2},\ldots,a_{k}\}$, where $0=a_{1}<a_{2}<\cdots<a_{k}$. Set $A'=A\setminus\{a_{1}\}$. So, $A'$ is a set of $k-1$ positive integers. First, let $\alpha=0$. By (\ref{equation1}) and (\ref{equation3}), we get
\[|\Sigma_{0}(A')|=|\Sigma^{0}(A')|=|\Sigma^{0}(A)|=\frac{(k-1)k}{2}+1.\]
Then, Theorem \ref{thm2.2} implies that $A'=a_{2}*[1,k-1]$. Hence, $A=a_{2}*[0,k-1]$.

Now, let $\alpha\geq 1$. By (\ref{equation2}) and (\ref{equation3}), we get
\[|\Sigma_{\alpha-1} (A')|=|\Sigma^{\alpha-1} (A')|=|\Sigma^{\alpha} (A)|=\frac{(k-1)k}{2}-\frac{(\alpha-1)\alpha}{2}+1.\]
Then, Theorem \ref{thm2.2} implies that $A'=a_{2}*[1,k-1]$. Hence, $A=a_{2}*[0,k-1]$. This completes the proof of the corollary.
\end{proof}

As a consequence of Theorem \ref{thm2.2} and Corollary \ref{cor2.3}, for $\alpha=0$, we get the following corollary.

\begin{corollary}[See \cite{nathu95}; Theorem 5]\label{cor2.4}
Let $k\geq 5$. If $A$ is a set of $k$ positive integers such that
\[|\Sigma (A)|=\frac{k(k+1)}{2}+1,\]
then $A=d*[1,k]$ for some positive integer $d$.

If $A$ is a set of $k$ nonnegative integers with $0\in A$ such that
\[|\Sigma (A)|=\frac{(k-1)k}{2}+1,\]
then $A=d*[0,k-1]$ for some positive integer $d$.
\end{corollary}

\section{Direct and inverse problems for subsequence sums}\label{subsequencesum}

\subsection{Direct problem}
Let $\mathcal{A}=(a_{1},a_{2},\ldots,a_{k})_{\bar{r}}$ be a finite sequence of $k$ distinct nonnegative integers with repetition $\bar{r}=(r_{1},r_{2},\ldots,r_{k})$, where $r_{i}\geq 1$ for $i=1,2,\ldots,k$. Let $0\leq \alpha\leq \sum_{i=1}^{k}r_{i}$ be an integer. If $\alpha=\sum_{i=1}^{k}r_{i}$, then $\Sigma_{\alpha} (\bar{r},\mathcal{A})=\{r_{1}a_{1}+r_{2}a_{2}+\cdots+r_{k}a_{k}\}$, and hence $|\Sigma_{\alpha} (\bar{r},\mathcal{A})|=1$. So, in the following theorem we assume that $0\leq \alpha\leq \sum_{i=1}^{k}r_{i}-1$.

\begin{theorem}\label{thm3.1}
Let $k\geq 1$. Let $\mathcal{A}=(a_{1},a_{2},\ldots,a_{k})_{\bar{r}}$ be a finite sequence of integers, where $0<a_{1}<a_{2}<\cdots<a_{k}$ and $\bar{r}=(r_{1},r_{2},\ldots,r_{k})$ with $r_{i}\geq 1$ for $i=1,2,\ldots,k$. Let $0\leq \alpha\leq \sum_{i=1}^{k}r_{i}-1$. Then there exists an integer $m\in[1,k]$ such that $\sum_{i=1}^{m-1}r_{i} \leq \alpha <\sum_{i=1}^{m}r_{i}$, and
\begin{equation}\label{eqn3.1}
|\Sigma_{\alpha} (\bar{r},\mathcal{A})|
\geq \sum_{i=1}^{k} ir_{i}-\sum_{i=1}^{m} ir_{i}+m\left(\sum_{i=1}^{m}r_{i}-\alpha\right)+1.
\end{equation}
The lower bound in (\ref{eqn3.1}) is best possible.
\end{theorem}

\begin{proof}
We prove the theorem by induction on $\alpha$. If $\alpha=\sum_{i=1}^{k}r_{i}-1$, then $\Sigma_{\alpha}(\bar{r},\mathcal{A})=\{r_{1}a_{1}+\cdots+r_{k-1}a_{k-1}+(r_{k}-1)a_{k},
r_{1}a_{1}+\cdots+r_{k-2}a_{k-2}+(r_{k-1}-1)a_{k-1}+r_{k}a_{k},\ldots,
(r_{1}-1)a_{1}+r_{2}a_{2}+\cdots+r_{k}a_{k},r_{1}a_{1}+\cdots+r_{k}a_{k}\}$. Hence
$|\Sigma_{\alpha} (\bar{r},\mathcal{A})|=k+1$. This satisfies (\ref{eqn3.1}).

Now, assume that (\ref{eqn3.1}) holds for $\alpha=\sum_{i=1}^{k}r_{i}-j$ for some $j=1,2,\ldots,\sum_{i=1}^{k}r_{i}-1$. We show that (\ref{eqn3.1}) also holds for $\alpha-1$. Note that, the smallest element of $\Sigma_{\alpha} (\bar{r},\mathcal{A})$ is $r_{1}a_{1}+\cdots+r_{m-1}a_{m-1}+\left(\alpha-\sum_{i=1}^{m-1}r_{i}\right)a_{m}$. So, the $\alpha$ distinct sums obtained by removing exactly one summand from $r_{1}a_{1}+\cdots+r_{m-1}a_{m-1}+\left(\alpha-\sum_{i=1}^{m-1}r_{i}\right)a_{m}$, all appear in $\Sigma_{\alpha-1} (\bar{r},\mathcal{A})$ and not in $\Sigma_{\alpha} (\bar{r},\mathcal{A})$. We also have $\Sigma_{\alpha} (\bar{r},\mathcal{A}) \subset \Sigma_{\alpha-1} (\bar{r},\mathcal{A})$. Thus, if $\sum_{i=1}^{m-1}r_{i}<\alpha<\sum_{i=1}^{m}r_{i}$, then

\begin{equation}\label{eqn3.2}
|\Sigma_{\alpha-1} (\bar{r},\mathcal{A})| \geq |\Sigma_{\alpha} (\bar{r},\mathcal{A})|+m.
\end{equation}
Therefore, by induction
\begin{align*}
|\Sigma_{\alpha-1} (\bar{r},\mathcal{A})|
&\geq \sum_{i=1}^{k} ir_{i}-\sum_{i=1}^{m} ir_{i}+m\left(\sum_{i=1}^{m}r_{i}-\alpha\right)+1+m\\
&= \sum_{i=1}^{k} ir_{i}-\sum_{i=1}^{m} ir_{i}+m\left(\sum_{i=1}^{m}r_{i}-(\alpha-1)\right)+1.
\end{align*}
Again, if $\alpha=\sum_{i=1}^{m-1}r_{i}$, then
\begin{equation}\label{eqn3.3}
|\Sigma_{\alpha-1} (\bar{r},\mathcal{A})| \geq |\Sigma_{\alpha} (\bar{r},\mathcal{A})|+m-1.
\end{equation}
Therefore, by induction
\begin{align*}
|\Sigma_{\alpha-1} (\bar{r},\mathcal{A})|
&\geq \sum_{i=1}^{k} ir_{i}-\sum_{i=1}^{m} ir_{i}+m\left(\sum_{i=1}^{m}r_{i}-\alpha\right)+1+m-1\\
&= \sum_{i=1}^{k} ir_{i}-\sum_{i=1}^{m-1} ir_{i}+(m-1)\left(\sum_{i=1}^{m-1}r_{i}-(\alpha-1)\right)+1.
\end{align*}
Hence, (\ref{eqn3.1}) holds for $\alpha=0,1,\ldots,\sum_{i=1}^{k}r_{i}-1$.

Next, we show that the lower bound in (\ref{eqn3.1}) is best possible. Let $k\geq 2$ and $\mathcal{A}=[1,k]_{\bar{r}}$, where $\bar{r}=(r_{1},r_{2},\ldots,r_{k})$. Then
\[\Sigma^{\alpha}(\bar{r},\mathcal{A})\subset \left[0,r_{k}k+\cdots+r_{m+1}(m+1)+
\left(\sum_{i=1}^{m}r_{i}-\alpha\right)m\right].\]
Therefore
\[|\Sigma_{\alpha}(\bar{r},\mathcal{A})|=|\Sigma^{\alpha}(\bar{r},\mathcal{A})|
\leq \sum_{i=1}^{k} ir_{i}-\sum_{i=1}^{m} ir_{i}+m\left(\sum_{i=1}^{m}r_{i}-\alpha\right)+1.\]
This together with (\ref{eqn3.1}) gives
\[|\Sigma_{\alpha}(\bar{r},\mathcal{A})|
= \sum_{i=1}^{k} ir_{i}-\sum_{i=1}^{m} ir_{i}+m\left(\sum_{i=1}^{m}r_{i}-\alpha\right)+1.\]
This completes the proof of the theorem.
\end{proof}

\begin{corollary}\label{cor3.2}
Let $k\geq 2$. Let $\mathcal{A}=(a_{1},a_{2},\ldots,a_{k})_{\bar{r}}$ be a finite sequence of integers, where $0=a_{1}<a_{2}<\cdots<a_{k}$ and $\bar{r}=(r_{1},r_{2},\ldots,r_{k})$ with $r_{i}\geq 1$ for $i=1,2,\ldots,k$. Let $0\leq \alpha\leq \sum_{i=1}^{k}r_{i}-1$. Then there exists an integer $m\in[1,k]$ such that $\sum_{i=1}^{m-1}r_{i} \leq \alpha <\sum_{i=1}^{m}r_{i}$, and
\begin{equation}\label{eqn3.4}
|\Sigma_{\alpha} (\bar{r},\mathcal{A})|
\geq \sum_{i=1}^{k} (i-1)r_{i}-\sum_{i=1}^{m} (i-1)r_{i}+(m-1)\left(\sum_{i=1}^{m}r_{i}-\alpha\right)+1.
\end{equation}
The lower bound in (\ref{eqn3.4}) is best possible.
\end{corollary}

\begin{proof}
Let $\mathcal{A}'=(a_{2},\ldots,a_{k})_{\bar{r'}}$, where $\bar{r'}=(r_{2},\ldots,r_{k})$. So, $\mathcal{A}'$ is a finite sequence of positive integers with $k-1$ distinct terms and repetition $\bar{r'}=(r_{2},\ldots,r_{k})$.

First, let $m=1$, i.e., $0\leq \alpha<r_{1}$. Then, it is easy to see that
\begin{equation}\label{eqn3.5}
\Sigma^{\alpha} (\bar{r}, \mathcal{A})=\Sigma^{0} (\bar{r'}, \mathcal{A}').
\end{equation}
Hence, by Theorem \ref{thm3.1}, we have
\begin{equation*}
|\Sigma_{\alpha} (\bar{r}, \mathcal{A})|= |\Sigma^{\alpha} (\bar{r}, \mathcal{A})|
= |\Sigma^{0}(\bar{r'}, \mathcal{A}')|
\geq \sum_{i=2}^{k} (i-1)r_{i}+1 = \sum_{i=1}^{k} (i-1)r_{i}+1.
\end{equation*}
This satisfies (\ref{eqn3.4}).

Now, let $m\geq 2$, i.e., $r_{1}\leq \alpha<\sum_{i=1}^{k}r_{i}$. Clearly, $\sum_{i=1}^{m-1}r_{i} \leq \alpha <\sum_{i=1}^{m}r_{i}$ implies that $\sum_{i=2}^{m-1}r_{i} \leq \alpha-r_{1} < \sum_{i=2}^{m}r_{i}$. Therefore
\begin{equation}\label{eqn3.6}
\Sigma^{\alpha} (\bar{r}, \mathcal{A})=\Sigma^{\alpha-r_{1}} (\bar{r'}, \mathcal{A}').
\end{equation}
Hence, by Theorem \ref{thm3.1}, we have
\begin{align*}
|\Sigma_{\alpha} (\bar{r}, \mathcal{A})|
&= |\Sigma^{\alpha} (\bar{r}, \mathcal{A})|\\
&= |\Sigma^{\alpha-r_{1}} (\bar{r'}, \mathcal{A}')|\\
&\geq \sum_{i=2}^{k} (i-1)r_{i}-\sum_{i=2}^{m} (i-1)r_{i}
+(m-1)\left(\sum_{i=2}^{m}r_{i}-(\alpha-r_{1})\right)+1\\
&= \sum_{i=1}^{k} (i-1)r_{i}-\sum_{i=1}^{m} (i-1)r_{i}+(m-1)\left(\sum_{i=1}^{m}r_{i}-\alpha\right)+1.
\end{align*}
This satisfies (\ref{eqn3.4}).

Next, we show that the lower bound in (\ref{eqn3.4}) is best possible. Let $k\geq 2$ and $\mathcal{A}=[0,k-1]_{\bar{r}}$, where $\bar{r}=(r_{1},r_{2},\ldots,r_{k})$. Then
\[\Sigma^{\alpha}(\bar{r},\mathcal{A})
\subset \left[0,r_{k}(k-1)+\cdots+r_{m+1}m+\left(\sum_{i=1}^{m}r_{i}-\alpha\right)(m-1)\right].\]
Therefore
\[|\Sigma_{\alpha}(\bar{r},\mathcal{A})|=|\Sigma^{\alpha}(\bar{r},\mathcal{A})|
\leq \sum_{i=1}^{k} (i-1)r_{i}-\sum_{i=1}^{m} (i-1)r_{i}+(m-1)\left(\sum_{i=1}^{m}r_{i}-\alpha\right)+1.\]
This together with (\ref{eqn3.4}) give
\[|\Sigma_{\alpha}(\bar{r},\mathcal{A})|=
\sum_{i=1}^{k} (i-1)r_{i}-\sum_{i=1}^{m} (i-1)r_{i}+(m-1)\left(\sum_{i=1}^{m}r_{i}-\alpha\right)+1.\]
This completes the proof of the corollary.
\end{proof}

As a particular case of Theorem \ref{thm3.1} and Corollary \ref{cor3.2}, for $\bar{r}=(r,r,\ldots,r)$, we obtain the following corollary.

\begin{corollary}\label{cor3.3}
Let $k\geq 2$, $r\geq 1$ and $0\leq \alpha <rk$. Let $m\in[1,k]$ be an integer such that $(m-1)r\leq\alpha<mr$. If $\mathcal{A}$ is a finite sequence of positive integers with $k$ distinct terms each repeating exactly $r$ times, then
\begin{equation}\label{eqn3.7}
|\Sigma_{\alpha} (r,\mathcal{A})|\geq r\left[\frac{k(k+1)}{2}-\frac{m(m+1)}{2}\right]+m(mr-\alpha)+1.
\end{equation}
If $\mathcal{A}$ is a finite sequence of nonnegative integers with $k$ distinct terms each repeating exactly $r$ times, and $0\in \mathcal{A}$, then
\begin{equation}\label{eqn3.8}
|\Sigma_{\alpha} (r,\mathcal{A})|\geq r\left[\frac{k(k+1)}{2}-\frac{m(m+1)}{2}\right]+(m-1)(mr-\alpha)+1.
\end{equation}
The lower bounds in (\ref{eqn3.7}) and (\ref{eqn3.8}) are best possible.
\end{corollary}

Again, as a consequence of Corollary \ref{cor3.3}, for $\alpha=0$, we obtain the following direct result of Mistri and Pandey \cite{mistri15} on usual subsequence sums.

\begin{corollary}[See \cite{mistri15}; Theorem 2.1]\label{cor3.4}
Let $k\geq 2$ and $r\geq 1$. Let $\mathcal{A}$ be a finite sequence of positive integers with $k$ distinct terms each repeating exactly $r$ times. Then
\begin{equation}\label{eqn3.9}
|\Sigma(r,\mathcal{A})|\geq r\binom{k+1}2+1.
\end{equation}

Let $\mathcal{A}$ be a finite sequence of nonnegative integers with $k$ distinct terms each repeating exactly $r$ times  and $0\in \mathcal{A}$. Then
\begin{equation}\label{eqn3.10}
|\Sigma(r,\mathcal{A})|\geq r\binom{k}2+1.
\end{equation}
The lower bounds in (\ref{eqn3.9}) and (\ref{eqn3.10}) are best possible.
\end{corollary}

\subsection{Inverse problem}

\begin{theorem}\label{thm3.5}
Let $k\geq 4$. Let $\bar{r}=(r_{1},r_{2},\ldots,r_{k})$, where $r_{i}\geq 1$ for $i=1,2,\ldots,k$, and let $0\leq \alpha\leq \sum_{i=1}^{k}r_{i}-2$. Let $m\in[1,k]$ be an integer such that $\sum_{i=1}^{m-1}r_{i} \leq \alpha <\sum_{i=1}^{m}r_{i}$. Let $\mathcal{A}=(a_{1},a_{2},\ldots,a_{k})_{\bar{r}}$ be a finite sequence of integers with  $0<a_{1}<a_{2}<\cdots<a_{k}$, and
\begin{equation}\label{eqn3.11}
|\Sigma_{\alpha}(\bar{r},\mathcal{A})|=
\sum_{i=1}^{k} ir_{i}-\sum_{i=1}^{m} ir_{i}+m\left(\sum_{i=1}^{m}r_{i}-\alpha\right)+1.
\end{equation}
Then \[\mathcal{A}=a_{1}* [1,k]_{\bar{r}}.\]
\end{theorem}

\begin{proof}
First, let $\alpha=\sum_{i=1}^{k}r_{i}-2$. If $r_{k}=1$, then $m=k-1$, otherwise $m=k$. Consider the subsequence sums $\Sigma^{\Sigma_{i=1}^{k}r_{i}-2}(\mathcal{A})$. If $r_{k}=1$ and (\ref{eqn3.11}) holds, i.e., $|\Sigma^{\Sigma_{i=1}^{k}r_{i}-2}(\mathcal{A})|=2k$, then
\begin{equation*}
\Sigma^{\Sigma_{i=1}^{k}r_{i}-2}(\mathcal{A})=\{0,a_{1},a_{2}\} \bigcup_{i=1}^{k-2} \{a_{i}+a_{i+1},a_{i}+a_{i+2}\} \bigcup \{a_{k-1}+a_{k}\}.
\end{equation*}
Similarly, if $r_{k}\geq 2$ and (\ref{eqn3.11}) holds, i.e., $|\Sigma^{\Sigma_{i=1}^{k}r_{i}-2}(\mathcal{A})|=2k+1$, then
\begin{equation*}
\Sigma^{\Sigma_{i=1}^{k}r_{i}-2}(\mathcal{A})=\{0,a_{1},a_{2}\} \bigcup_{i=1}^{k-2} \{a_{i}+a_{i+1},a_{i}+a_{i+2}\} \bigcup \{a_{k-1}+a_{k},2a_{k}\}.
\end{equation*}
By the same argument as used in the proof of Theorem \ref{thm2.2}, for the subset sums $\Sigma^{k-2}(A)$, one can obtain that
$a_{k}-a_{k-1}=\cdots=a_{2}-a_{1}=a_{1}$. Thus, $\mathcal{A}=a_{1}*[1,k]_{\bar{r}}$. Hence, the theorem holds for $\alpha=\sum_{i=1}^{k}r_{i}-2$.

Now, suppose that the theorem holds for $\alpha=\sum_{i=1}^{k}r_{i}-j$ for some $j=2,3,\ldots,\sum_{i=1}^{k}r_{i}-1$. We show that the theorem also holds for $\alpha-1$. If $\sum_{i=1}^{m-1}r_{i}<\alpha<\sum_{i=1}^{m}r_{i}$ and (\ref{eqn3.11}) holds for $\alpha-1$, i.e.,
\[|\Sigma_{\alpha-1} (\bar{r},\mathcal{A})|
=\sum_{i=1}^{k} ir_{i}-\sum_{i=1}^{m} ir_{i}+m\left(\sum_{i=1}^{m}r_{i}-(\alpha-1)\right)+1,\]
then by (\ref{eqn3.2}), we get
\begin{align*}
|\Sigma_{\alpha} (\bar{r},\mathcal{A})|
&\leq |\Sigma_{\alpha-1} (\bar{r},\mathcal{A})|-m \\
&= \sum_{i=1}^{k} ir_{i}-\sum_{i=1}^{m} ir_{i}+m\left(\sum_{i=1}^{m}r_{i}-(\alpha-1)\right)+1-m \\
&= \sum_{i=1}^{k} ir_{i}-\sum_{i=1}^{m} ir_{i}+m\left(\sum_{i=1}^{m}r_{i}-\alpha\right)+1.
\end{align*}
If $\alpha=\sum_{i=1}^{m-1}r_{i}$ and (\ref{eqn3.11}) holds for $\alpha-1$, i.e.,
\[|\Sigma_{\alpha-1} (\bar{r},\mathcal{A})|
=\sum_{i=1}^{k} ir_{i}-\sum_{i=1}^{m-1} ir_{i}+(m-1)\left(\sum_{i=1}^{m-1}r_{i}-(\alpha-1)\right)+1,\]
then by (\ref{eqn3.3}), we get
\begin{align*}
|\Sigma_{\alpha} (\bar{r},\mathcal{A})|
&\leq |\Sigma_{\alpha-1} (\bar{r},\mathcal{A})|-(m-1) \\
&=\sum_{i=1}^{k}ir_{i}-\sum_{i=1}^{m-1}ir_{i}+(m-1)\left(\sum_{i=1}^{m-1}r_{i}-(\alpha-1)\right)+1-(m-1) \\
&= \sum_{i=1}^{k} ir_{i}-\sum_{i=1}^{m} ir_{i}+m\left(\sum_{i=1}^{m}r_{i}-\alpha\right)+1.
\end{align*}
In both cases, we get
\[|\Sigma_{\alpha} (\bar{r},\mathcal{A})|
\leq \sum_{i=1}^{k} ir_{i}-\sum_{i=1}^{m} ir_{i}+m\left(\sum_{i=1}^{m}r_{i}-\alpha\right)+1.\]
This together with (\ref{eqn3.1}) give
$|\Sigma_{\alpha}(r,\mathcal{A})|
=\sum_{i=1}^{k} ir_{i}-\sum_{i=1}^{m} ir_{i}+m\left(\sum_{i=1}^{m}r_{i}-\alpha\right)+1$.
Hence, by induction $\mathcal{A}=a_{1}*[1,k]_{\bar{r}}$. This completes the proof of the theorem.
\end{proof}

\begin{remark}\label{remark3}
The following are some sequences for which (\ref{eqn3.11}) holds (i.e., extremal sequences), but they are not of the form  $d*[1,k]_{\bar{r}}$.
\begin{enumerate}[label=(\roman*)]
  \item Let $\mathcal{A}=(a_{1},a_{2},\ldots,a_{k})_{\bar{r}}$ be a finite sequence of integers with  $0<a_{1}<a_{2}<\cdots<a_{k}$ and $\bar{r}=(r_{1},r_{2},\ldots,r_{k})$, where $r_{i}\geq 1$ for $i=1,2,\ldots,k$. If $\alpha=\sum_{i=1}^{k}r_{i}$, then $\Sigma_{\alpha}(\bar{r},\mathcal{A})=\{r_{1}a_{1}+\cdots+r_{k}a_{k}\}$, and hence $|\Sigma_{\alpha} (\bar{r},\mathcal{A})|=1$. Similarly, if $\alpha=\sum_{i=1}^{k}r_{i}-1$, then $\Sigma_{\alpha}(\bar{r},\mathcal{A})=\{r_{1}a_{1}+\cdots+r_{k-1}a_{k-1}+(r_{k}-1)a_{k}, r_{1}a_{1}+\cdots+r_{k-2}a_{k-2}+(r_{k-1}-1)a_{k-1}+r_{k}a_{k},\ldots, (r_{1}-1)a_{1}+r_{2}a_{2}+\cdots+r_{k}a_{k},r_{1}a_{1}+\cdots+r_{k}a_{k}\}$, and hence $|\Sigma_{\alpha} (\bar{r},\mathcal{A})|=k+1$. In both the cases (\ref{eqn3.11}) holds. Thus, every sequence $\mathcal{A}$ is an extremal sequence for $\alpha=\sum_{i=1}^{k}r_{i}~\text{and}~\sum_{i=1}^{k}r_{i}-1$.

  \item Let $\mathcal{A}=(a_{1},a_{2})_{\bar{r}}$, where $0<a_{1}<a_{2}$ and $\bar{r}=(r_{1},r_{2})$, $r_{i}\geq 1$. Since, the cases $\alpha=r_{1}+r_{2}~\text{and}~r_{1}+r_{2}-1$ are covered in (1) we let  $\alpha\leq r_{1}+r_{2}-2$.

      First, let $\alpha=r_{1}+r_{2}-2$. Clearly, $m=1$ if $r_{2}=1$, otherwise $m=2$. If $r_{2}=1$ and (\ref{eqn3.11}) holds, then $\Sigma^{r_{1}+r_{2}-2}(\bar{r},\mathcal{A})=\{0,a_{1},a_{2},a_{1}+a_{2}\}$. Similarly, if $r_{2}>1$ and (\ref{eqn3.11}) holds, then $\Sigma^{r_{1}+r_{2}-2}(\bar{r},\mathcal{A})=\{0,a_{1},a_{2},a_{1}+a_{2},2a_{2}\}$. Thus, if $r_{1}=1$, then every sequence $\mathcal{A}=(a_{1},a_{2})_{\bar{r}}$ with $\bar{r}=(1,r_{2})$, $r_{2}\geq 1$ is an extremal sequence. If $r_{1}>1$, then $a_{1}<2a_{1}<a_{1}+a_{2}$ implies that $a_{2}=2a_{1}$. Thus, in this case $\mathcal{A}=(a_{1},2a_{1})_{\bar{r}}=a_{1}*[1,2]_{\bar{r}}$.

      Now, let $\alpha<r_{1}+r_{2}-2$. The induction argument on $\alpha$ implies that every sequence $\mathcal{A}=(a_{1},a_{2})_{\bar{r}}$ with $\bar{r}=(1,r_{2})$, $r_{2}\geq 1$ is an extremal sequence. If $r_{1}>1$, then $\mathcal{A}=a_{1}*[1,2]_{\bar{r}}$.

  \item Let $\mathcal{A}=(a_{1},a_{2},a_{3})_{\bar{r}}$, where $0<a_{1}<a_{2}<a_{3}$ and $\bar{r}=(r_{1},r_{2},r_{3})$. Since, the cases $\alpha=r_{1}+r_{2}+r_{3}~\text{and}~r_{1}+r_{2}+r_{3}-1$ are covered in (1) we let  $\alpha\leq r_{1}+r_{2}+r_{3}-2$.

      First, let $\alpha=r_{1}+r_{2}+r_{3}-2$. Clearly, $m=2$ if $r_{3}=1$, otherwise $m=3$. If $r_{3}=1$ and (\ref{eqn3.11}) holds, then $\Sigma^{r_{1}+r_{2}+r_{3}-2}(\bar{r},\mathcal{A})=\{0,a_{1},a_{2},a_{1}+a_{2},a_{1}+a_{3},a_{2}+a_{3}\}$. Similarly, if $r_{3}>1$ and (\ref{eqn3.11}) holds, then $\Sigma^{r_{1}+r_{2}+r_{3}-2}(\bar{r},\mathcal{A})=\{0,a_{1},a_{2},a_{1}+a_{2},a_{1}+a_{3},a_{2}+a_{3},
      2a_{3}\}$. Since $a_{2}<a_{3}<a_{1}+a_{3}$ and $a_{2}<a_{1}+a_{2}<a_{1}+a_{3}$, we get $a_{3}=a_{1}+a_{2}$. Thus, $\mathcal{A}=(a_{1},a_{2},a_{1}+a_{2})_{\bar{r}}$.

      Now, if $r_{1}>1$, then $a_{1}<2a_{1}<a_{1}+a_{2}$ implies that $a_{2}=2a_{1}$. Similarly, if $r_{2}>1$, then $a_{1}+a_{2}<2a_{2}<a_{2}+a_{3}$ implies that $2a_{2}=a_{1}+a_{3}$. In other words $a_{3}-a_{2}=a_{2}-a_{1}$. Thus, if $r_{1}>1$ or $r_{2}>1$, then $\mathcal{A}=a_{1}*[1,3]_{\bar{r}}$. Hence, if  $r_{1}=r_{2}=1$, then $\mathcal{A}=(a_{1},a_{2},a_{1}+a_{2})_{\bar{r}}$ is an extremal sequence which is not of the form $d*[1,3]_{\bar{r}}$.

      Now, let $0\leq \alpha<r_{1}+r_{2}+r_{3}-2$. The induction argument on $\alpha$ implies that every sequence $\mathcal{A}=(a_{1},a_{2},a_{1}+a_{2})_{\bar{r}}$ with $\bar{r}=(1,1,r_{3})$, $r_{3}\geq 1$ is an extremal sequence. If $r_{1}>1$ or $r_{2}>1$, then $\mathcal{A}=a_{1}*[1,3]_{\bar{r}}$.
\end{enumerate}
\end{remark}

\begin{corollary}\label{cor3.6}
Let $k\geq 5$. Let $\bar{r}=(r_{1},r_{2},\ldots,r_{k})$, where $r_{i}\geq 1$ for $i=1,2,\ldots,k$, and let $0\leq \alpha\leq \sum_{i=1}^{k}r_{i}-1$. Let $m\in[1,k]$ be an integer such that $\sum_{i=1}^{m-1}r_{i} \leq \alpha <\sum_{i=1}^{m}r_{i}$. Let $\mathcal{A}=(a_{1},a_{2},\ldots,a_{k})_{\bar{r}}$ be a nonempty sequence of integers with $0=a_{1}<a_{2}<\cdots<a_{k}$, and
\begin{equation}\label{equation4}
|\Sigma_{\alpha}(\bar{r},\mathcal{A})|=
\sum_{i=1}^{k} (i-1)r_{i}-\sum_{i=1}^{m} (i-1)r_{i}+(m-1)\left(\sum_{i=1}^{m}r_{i}-\alpha\right)+1.
\end{equation}
Then \[\mathcal{A}=a_{2}* [0,k-1]_{\bar{r}}.\]
\end{corollary}

\begin{proof}
Let $\mathcal{A}'=(a_{2},\ldots,a_{k})_{\bar{r'}}$, where $\bar{r'}=(r_{2},\ldots,r_{k})$. So, $\mathcal{A}'$ is a finite sequence of positive integers with $k-1$ distinct terms and repetition $\bar{r'}=(r_{2},\ldots,r_{k})$.

First, let $m=1$, i.e., $0\leq \alpha<r_{1}$. By (\ref{eqn3.5}) and (\ref{equation4}), we get
\[|\Sigma_{0} (\bar{r'}, \mathcal{A}')|
=|\Sigma^{0} (\bar{r'}, \mathcal{A}')|
=|\Sigma^{\alpha} (\bar{r}, \mathcal{A})|
=\sum_{i=1}^{k} (i-1)r_{i}+1=\sum_{i=2}^{k} (i-1)r_{i}+1.\]
Then, Theorem \ref{thm3.5} (for $\alpha=0$) implies that $\mathcal{A'}=a_{2}*[1,k-1]_{\bar{r'}}$. Hence, $\mathcal{A}=a_{2}*[0,k-1]_{\bar{r}}$.

Now, let $m\geq 2$, i.e., $r_{1}\leq \alpha<\sum_{i=1}^{k}r_{i}$. By (\ref{eqn3.6}) and (\ref{equation4}), we get
\begin{align*}
|\Sigma_{\alpha-r_{1}} (\bar{r'}, \mathcal{A}')|
&= |\Sigma^{\alpha-r_{1}} (\bar{r'}, \mathcal{A}')|\\
&= |\Sigma^{\alpha} (\bar{r}, \mathcal{A})|\\
&= \sum_{i=1}^{k} (i-1)r_{i}-\sum_{i=1}^{m} (i-1)r_{i}+(m-1)\left(\sum_{i=1}^{m}r_{i}-\alpha\right)+1\\
&= \sum_{i=2}^{k} (i-1)r_{i}-\sum_{i=2}^{m} (i-1)r_{i}+(m-1)\left(\sum_{i=2}^{m}r_{i}-(\alpha-r_{1})\right)+1.
\end{align*}
Then, Theorem \ref{thm3.5} (for $\alpha-r_{1}$) implies that $\mathcal{A'}=a_{2}*[1,k-1]_{\bar{r'}}$. Hence, $\mathcal{A}=a_{2}*[0,k-1]_{\bar{r}}$. This completes the proof of the corollary.
\end{proof}

\begin{remark}\label{remark4}
As a corollary of Remark \ref{remark3}, we get the following sequences which are extremal sequences but not of the form $d*[0,k-1]_{\bar{r}}$.
\begin{enumerate}[label=(\roman*)]
  \item Every sequence $\mathcal{A}=(0,a_{1},\ldots,a_{k-1})_{\bar{r}}$ with $0<a_{1}<\cdots<a_{k-1}$ and $\bar{r}=(r_{0},r_{1},\ldots,r_{k-1})$ is an extremal sequence for $\alpha=\sum_{i=1}^{k}r_{i}~\text{and}~\sum_{i=1}^{k}r_{i}-1$.

  \item Every sequence $\mathcal{A}=(0,a_{1},a_{2})_{\bar{r}}$ with $0<a_{1}<a_{2}$ and $\bar{r}=(r_{0},1,r_{2})$ is an extremal sequence for every $\alpha$.

  \item Every sequence $\mathcal{A}=(0,a_{1},a_{2},a_{1}+a_{2})_{\bar{r}}$ with $0<a_{1}<a_{2}$ and $\bar{r}=(r_{0},1,1,r_{3})$ is an extremal sequence for every $\alpha$.
\end{enumerate}
\end{remark}

As a particular case of Theorem \ref{thm3.5} and Corollary \ref{cor3.6}, for $\bar{r}=(r,r,\ldots,r)$, we get the following corollary.

\begin{corollary}\label{cor3.7}
Let $k\geq 5$, $r\geq 1$ and $0\leq \alpha \leq rk-2$. Let $1\leq m\leq k$ be an integer such that $(m-1)r\leq\alpha< mr$. If $\mathcal{A}$ is a finite sequence of positive integers with $k$ distinct terms each repeating exactly $r$ times such that
\[|\Sigma_{\alpha} (r,\mathcal{A})|=r\left[\frac{k(k+1)}{2}-\frac{m(m+1)}{2}\right]+m(mr-\alpha)+1,\] then $\mathcal{A}=d*[1,k]_{r}$ for some positive integer $d$.\\
If $\mathcal{A}$ is a finite sequence of nonnegative integers with $k$ distinct terms each repeating exactly $r$ times and $0\in \mathcal{A}$, such that
\[|\Sigma_{\alpha} (r,\mathcal{A})|=r\left[\frac{k(k+1)}{2}-\frac{m(m+1)}{2}\right]+(m-1)(mr-\alpha)+1,\] then $\mathcal{A}=d*[0,k-1]_{r}$ for some positive integer $d$.\\
\end{corollary}

Again, as a consequence Corollary \ref{cor3.7}, for $\alpha=0$, we obtain the following inverse result of Mistri and Pandey \cite{mistri15} on regular subsequence sums.

\begin{corollary}[See \cite{mistri15}; Theorem 2.3]\label{cor3.8}
Let $k\geq 5$ and $r\geq 1$. If $\mathcal{A}$ is a finite sequence of positive integers with $k$ distinct terms each repeating exactly $r$ times such that
\[|\Sigma (r,\mathcal{A})|=\frac{rk(k+1)}{2}+1,\]
then $\mathcal{A}=d*[1,k]_{r}$ for some positive integer $d$.

If $\mathcal{A}$ is a finite sequence of nonnegative integers with $k$ distinct terms each repeating exactly $r$ times and $0\in \mathcal{A}$ such that
\[|\Sigma (r,\mathcal{A})|=\frac{r(k-1)k}{2}+1,\]
then $\mathcal{A}=d*[0,k-1]_{r}$ for some positive integer $d$.
\end{corollary}

\bibliographystyle{amsplain}


\begin{thebibliography}{10}
\bibitem{eric17} \'{E}. Balandraud, Addition theorems in $F_{p}$ via the polynomial method, ArXiv:1702.06419v1.

\bibitem{bollobas} B. Bollob\'{a}s, I. Leader, The number of $k$-sums modulo $k$, {\it J. Number Theory} {\bf 78} (1999) 27--35.

\bibitem{cziszter}  K. Cziszter, Improvements of the Noether bound for polynomial invariants of finite groups, PhD thesis, CEU Budapest, 2012.

\bibitem{erdos} P. Erd\H{o}s, A. Ginzburg, A. Ziv, Theorem in the additive number theory, {\it Bull. Res. Counc.} {\bf 10F} (1961) 41--43.


\bibitem{gao04} W.D. Gao, I.Z. Ruzsa, R. Thangadurai, Olson’s constant for the group $\mathbb{Z}_{p}\oplus \mathbb{Z}_{p}$, {\it J. Combin. Theory Ser. A} {\bf 107} (2004) 49--67.

\bibitem{gao08} W.D. Gao, R. Thangadurai, J. Zhuang, Addition theorems on the cyclic groups of order $p^{l}$, {\it Discrete Math.} {\bf 308} (10) (2008) 2030--2033.

\bibitem{girard} B. Girard, W. Schmid, Direct zero-sum problems for certain groups of rank three, {\it J. Number Theory} {\bf 197} (2019) 297--316.

\bibitem{grynkiewicz1}  D. Grynkiewicz, On a partition analog of the Cauchy–Davenport theorem, {\it Acta Math. Hungar.} {\bf 107} (2005) 167--181.

\bibitem{grynkiewicz2}  D. Grynkiewicz, On a conjecture of Hamidoune for subsequence sums, {\it Integers} {\bf 5} (2) (2005) A07.

\bibitem{jiang} X.W. Jiang, Y.L. Li, On the cardinality of subsequence sums, {\it Int. J. Number Theory} {\bf 14} (2018) 661--668.

\bibitem{hamidoune}  Y.O. Hamidoune, Subsequence sums, {\it Combin. Probab. Comput.} {\bf 12} (2003) 413--425.

\bibitem{mistri15} R.K. Mistri, R.K. Pandey, O. Prakash, Subsequence sums: direct and inverse problems, {\it J. Number Theory} {\bf 148} (2015) 235--256.

\bibitem{mistri16} R.K. Mistri, R.K. Pandey, The direct and inverse theorems on integer subsequence sums revisited, {\it Integers} {\bf 16} (2016) A32.


\bibitem{nathu95} M.B. Nathanson, Inverse theorems for subset sums, {\it Trans. Amer. Math. Soc.} {\bf 347} (1995) 1409--1418.

\bibitem{ordaz} O. Ordaz, A. Philipp, I. Santos, W. Schmid, On the Olson and the strong Davenport con-
stants, {\it J. Th\'{e}or. Nombres Bordeaux} {\bf 23} (2011) 715--750.

\bibitem{schmid} W. Schmid, Restricted inverse zero-sum problems in groups of rank two, {\it Q. J. Math.}
     {\bf 63} (2012) 477--487.
\end{thebibliography}

\end{document}